\documentclass[11pt]{amsart}

\usepackage{amsmath,amssymb,amsthm,mathtools}
\mathtoolsset{showonlyrefs}
\usepackage{tikz}
\usepackage[margin=1.15in]{geometry}
\usepackage[colorlinks=true,citecolor=blue,linkcolor=blue,urlcolor=blue]{hyperref}

\numberwithin{equation}{section}

\newtheorem{theorem}{Theorem}[section]
\newtheorem{proposition}[theorem]{Proposition}
\newtheorem{lemma}[theorem]{Lemma}
\newtheorem{corollary}[theorem]{Corollary}
\theoremstyle{definition}

\theoremstyle{plain}
\newtheorem{remark}[theorem]{Remark}
\newcommand{\R}{\mathbb R}
\newcommand{\Sph}{\mathbb S}

\begin{document}
	\title[Alexandrov-Type Rigidity on Rotational Supports]
	{Alexandrov-Type Rigidity for Minimal Capillary Hypersurfaces on Rotational Supports}

	\author{Caiyan Li}
	\address{(C.L.) School of Mathematical Sciences, Xiamen University, 361005, Xiamen, P.~R.~China}
	\email{caiyanli@xmu.edu.cn} 
	\thanks{C.L. is supported by NSFC (Grant No. 12501274).}
	\author{Chao Xia}
	\address{(C.X.) School of Mathematical Sciences, Xiamen University, 361005, Xiamen, P.~R.~China}
	\email{chaoxia@xmu.edu.cn}
	\thanks{C.X. is supported by NSFC (Grant No. 12271449, 12671069, 12526203, 12526102) and the Natural Science Foundation of Fujian Province of China (Grant No. 2024J011008).}
	\date{}
	
	\begin{abstract}
		In this paper, we prove that, under suitable conditions on the generating profile, every connected, compact embedded minimal capillary hypersurface supported on a one-ended rotational hypersurface in the Euclidean space is a horizontal slice. 
		We also construct a smooth one-ended rotational support carrying a non-horizontal planar capillary $n$-ball, showing that the main slope condition cannot in general be omitted. For obtuse contact angles, we further obtain a catenoid-band classification under a reverse profile inequality.
	\end{abstract}
	
	\maketitle

	\section{Introduction}
	
	Alexandrov's celebrated soap bubble theorem says that every closed embedded hypersurface of constant mean curvature in Euclidean space is a round sphere \cite{Alexandrov}. Alexandrov developed the method of moving planes to prove this theorem. An alternative method that is based on a Minkowski integral formula and the sharp Heintze-Karcher inequality was given by Ros \cite{Ros}, which was inspired by Reilly \cite{Reilly}. There are similar works in space forms \cite{Korevaar, MontielRos}. 
	
	Wente \cite{Wente} studied the equilibrium shape of liquid drops and formulated it as capillary surfaces supported on a hyperplane. He showed  the Alexandrov-type rigidity result that any compact embedded hypersurface of constant mean curvature with capillary boundary in a Euclidean half-space is a spherical cap. Recently, Jia, Wang, Zhang and the second author \cite{JiaWangXiaZhangWedge} gave a Ros-type proof of Wente's result, based on sharp Heintze-Karcher inequality for capillary surfaces. This method has been used to prove the Alexandrov-type rigidity result  for capillary hypersurfaces in various settings \cite{ JiaWangXiaZhangAnisotropic, JiaWangXiaZhangJMS, WangXiaZermelo, HuWeiXiaZhou}.
	These studies have in common that the supporting hypersurfaces have sufficiently strong geometry, for example, they are flat or umbilical surfaces in space forms. 
	It is because the methods depend on strong geometric properties of the support. A horizontal hyperplane is preserved by reflection across every vertical hyperplane, while a sphere is umbilical and carries the conformal vector fields needed for the Minkowski formulas. 
	For a support with less symmetry, neither the required moving reflections nor the special boundary identities are generally available. This is therefore natural to ask how much symmetry of the support is actually needed for an Alexandrov-type theorem.

	In the closed case,
	Brendle \cite{Brendle} showed that Alexandrov-type rigidity can still survive when
	strong ambient symmetry can be replaced by a distinguished radial structure. He considered a family of sub-static warped product spaces
	\[
	(M,g)=(N\times[0,\bar r),\,dr^2+h(r)^2g_N),
	\]
	where $N$ is a closed manifold and $h$ is warping factor, and showed that
	any closed embedded CMC hypersurface in such spaces is a slice. See also \cite{BE} for similar result on constant Weingarten curvature surfaces. Li and the second author \cite{LiXia19} established a Reilly-type formula for sub-static manifolds and gave an alternative proof of Brendle's result, see also \cite{BorghiniFogagnolo24}.
	The basic model for sub-static warped product spaces is the Schwarzschild manifold; more generally, the theorem applies to the de Sitter--Schwarzschild and Reissner--Nordstr\"om manifolds. Thus Brendle's theorem does not require a highly symmetric ambient space; it uses only a one-dimensional radial structure determined by $h$.

A related extrinsic model appears in the work of Eichmair--Koerber \cite{EK}. For an asymptotically flat support surface $S\subset\R^3$ with nonnegative mean curvature and an outermost free-boundary minimal surface $D$, they studied free-energy-minimizing minimal capillary surfaces supported on $S$ to prove an extrinsic Penrose conjecture proposed by Huisken \cite[p.~38]{Vol15}. The conjecture states that the exterior mass $m_{ext}$ of $S$ satisfies
\begin{equation}
	m_{ext}\geq\sqrt{\frac{|D|}{\pi}},
	\label{eq}
\end{equation}
an extrinsic analogue of the Riemannian Penrose inequality. The Riemannian Penrose inequality bounds the ADM mass of an asymptotically flat three-manifold with nonnegative scalar curvature from below in terms of the area of its outermost minimal boundary, with equality characterized by the spatial Schwarzschild manifold \cite{HuiskenIlmanen,Bray}. In Huisken's extrinsic setting, $D$ plays the role of the horizon, while the exterior mass is determined by the asymptotic geometry of $S$. Eichmair--Koerber \cite{EK} proved that the associated free-energy mass is nondecreasing as the contact angle increases. Comparing its value at $D$ with its limit at infinity gives \eqref{eq}, and equality holds if and only if the unbounded component of $S\setminus\partial D$ is a half-catenoid. In a further work \cite{EK2}, they have characterized all compact embedded stable minimal capillary surfaces with capillary angle close to either $0$ or $\pi$ that are supported on a complete embedded minimal surface with finite total curvature.  More recently, the first author \cite{LiHorizonFree} proved the extrinsic Penrose inequality without assuming an outermost free-boundary minimal surface; its positive-area equality model is the catenoid.

The above gives a parallel analogue between Schwarzschild space as the intrinsic model and the half-catenoid as the extrinsic model. In this paper, motivated by Brendle's result under radial symmetry and by the equality case of Eichmair--Koerber, we study compact embedded minimal capillary hypersurfaces supported on the half-catenoid. More generally, we consider one-ended rotational supports $r=\rho(z)$ in Euclidean space and give conditions on $\rho$ under which every such hypersurface is a union of horizontal slices. In this extrinsic analogue of Brendle's setting, the profile $\rho$ plays the role of the warping function $h$.

	To state the results, fix $n\geq2$, $0<L\leq\infty$, and a positive function $\rho\in C^2([0,L))$ whose meridian is proper and complete as $z\uparrow L$. Define
	\[
	S_\rho=\{(\rho(z)\omega,z)\in \mathbb{R}^{n+1}:\omega\in\Sph^{n-1},\ 0\leq z<L\},\] 
    \[M_\rho=\{(y,z)\in \mathbb{R}^{n+1}:|y|\leq\rho(z),\ 0\leq z<L\}.
	\]
	We write
	\[
	D_0=\{(y,0):|y|\leq\rho(0)\},\qquad N_{S_\rho}=\frac{(\omega,-\rho')}{\sqrt{1+\rho'^2}},
	\]
	where $D_0$ is the bottom $n$-ball of $M_\rho$ and $N_{S_\rho}$ is the outward unit normal of $M_\rho$ along $S_\rho$. Thus $\partial M_\rho=D_0\cup S_\rho$, $S_\rho$ has one noncompact end, and $M_\rho$ has no upper cap.
	
	Let $\Sigma^n\subset\overline{M_\rho}$ be a smooth, compact, properly embedded minimal \textit{$\theta$-capillary} hypersurface supported on $S_\rho$, with $0<\theta<\pi$ and
	\[
	\Sigma\setminus\partial\Sigma\subset\operatorname{int}M_\rho,\qquad \varnothing\neq\partial\Sigma\subset S_\rho\setminus\partial D_0.
	\]
	Each component of $\Sigma$ is assumed two-sided. Choose a unit normal $\nu_\Sigma$ on each component satisfying
	\begin{equation}
		\langle\nu_\Sigma,N_{S_\rho}\rangle=-\cos\theta.
		\label{eq:angle}
	\end{equation}
	Throughout, homology is taken with $\mathbb Z_2$ coefficients.
	
	For a connected $\Sigma$, let $W_\Sigma\subset\widehat S_\rho:=D_0\cup S_\rho$ be the compact region bounded by $\partial\Sigma$, and let $\Omega_\Sigma\subset M_\rho$ be the bounded region enclosed by $\Sigma\cup W_\Sigma$. See Figure~\ref{fig:geometric-regions}.
	
	\begin{figure}[ht]
		\centering
		\begin{tikzpicture}[x=1.35cm,y=1.02cm,>=stealth,font=\small]
			\fill[black!8]
			(-1,0)
			.. controls (-1.02,0.9) and (-1.18,1.8) .. (-1.4,2.5)
			--(1.4,2.5)
			.. controls (1.18,1.8) and (1.02,0.9) .. (1,0)--cycle;
			\draw[dashed] (0,-0.15)--(0,4.75) node[above] {$z$};
			\draw[line width=0.8pt]
			(1,0).. controls (1.02,0.9) and (1.18,1.8) .. (1.4,2.5)
			.. controls (1.65,3.2) and (1.95,3.9) .. (2.25,4.55);
			\draw[line width=0.8pt]
			(-1,0).. controls (-1.02,0.9) and (-1.18,1.8) .. (-1.4,2.5)
			.. controls (-1.65,3.2) and (-1.95,3.9) .. (-2.25,4.55);
			\draw[line width=0.8pt] (-1,0)--(1,0);
			\draw[line width=1.5pt]
			(-1.4,2.5).. controls (-1.18,1.8) and (-1.02,0.9) .. (-1,0)
			--(1,0).. controls (1.02,0.9) and (1.18,1.8) .. (1.4,2.5);
			\draw[line width=1.2pt] (-1.4,2.5)--(1.4,2.5);
			\fill (-1.4,2.5) circle (1.2pt);
			\fill (1.4,2.5) circle (1.2pt);
			\node[below] at (0,0) {$D_0$};
			\node[right] at (1.88,3.75) {$S_\rho$};
			\node[above] at (0,2.5) {$\Sigma$};
			\node at (0,1.15) {$\Omega_\Sigma$};
			\node[left] at (-1.28,1.18) {$W_\Sigma$};
			\node[right] at (1.28,1.18) {$W_\Sigma$};
			\draw[->] (2.55,2.85) node[right,align=left] {$\widehat S_\rho=D_0\cup S_\rho$}
			-- (1.92,3.45);
		\end{tikzpicture}
		\caption{A meridian plane through the $z$-axis, shown for a horizontal minimal $n$-ball $\Sigma$. The two outer curves form the section of $S_\rho$, and the bottom segment is $D_0$. The thick part of $\widehat S_\rho$ is $W_\Sigma$, and the shaded region is $\Omega_\Sigma$.}
		\label{fig:geometric-regions}
	\end{figure}
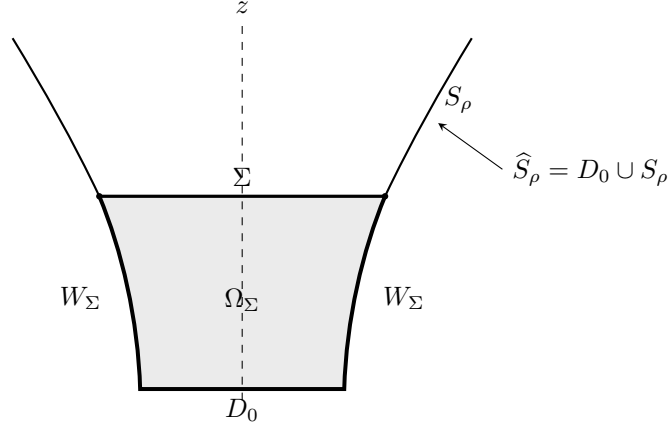
	
	Two obstructions must be excluded. First, $\rho'$ may cross the contact-angle slope $-\cot\theta$ from above to below. This is ruled out by
	\begin{equation}
		0<a<b<L,\quad \rho'(a)>-\cot\theta\quad\Longrightarrow\quad \rho'(b)\geq-\cot\theta.
		\label{eq:slope-condition}
	\end{equation}
	Thus, once $\rho'$ exceeds $-\cot\theta$, it cannot subsequently fall below this value.
	
	Second, two horizontal sections of $S_\rho$ may bound a catenoid band, by which we mean a compact annular part of a rotational catenoid bounded by two horizontal spheres. For $\theta\neq\pi/2$, define
	\begin{equation}
		\ell_\theta:=\sup\left\{b\in[0,L):\rho'(b)<-\cot\theta\right\},
		\label{eq:ell-theta}
	\end{equation}
	with the convention that $\ell_\theta=0$ if the set is empty. Define
	\begin{equation}
		A_\theta=\frac{\rho^{n-1}(1+\tan\theta\,\rho')}{\sqrt{1+\rho'^2}},\qquad
		B_\theta=\frac{\rho^{n-1}(1-\tan\theta\,\rho')}{\sqrt{1+\rho'^2}}.
		\label{eq:profile-barriers}
	\end{equation}
	
	We rule out the second obstruction by the strict profile inequality
	\begin{equation}
		A_\theta(b)<B_\theta(a)\qquad(0<a<b<\ell_\theta).
		\label{eq:two-point-barrier}
	\end{equation}
	Our main result is the following Alexandrov-type rigidity theorem.
	
	\begin{theorem}
		\label{thm:intro-disk-rigidity}
		Let $0<\theta<\pi$, and let $\Sigma\subset\overline{M_\rho}$ be a compact embedded minimal $\theta$-capillary hypersurface supported on $S_\rho$. Assume \eqref{eq:slope-condition}.
		\begin{enumerate}
			\item[(i)] Every component $\Sigma_0$ of $\Sigma$ homologous to $D_0$ relative to $S_\rho$ is a horizontal slice.
			\item[(ii)] If $\theta=\pi/2$, every component of $\Sigma$ is a horizontal slice.
			\item[(iii)] If $\theta\neq\pi/2$ and, in addition, \eqref{eq:two-point-barrier} holds, every component of $\Sigma$ is a horizontal slice.
		\end{enumerate}
	\end{theorem}
	
	The half-catenoid satisfies the hypotheses of Theorem~\ref{thm:intro-disk-rigidity} and gives the following model case.
	
	\begin{corollary}
		\label{cor:half-catenoid}
		Let $0<\theta<\pi$, and let $\Sigma\subset\overline{M_\rho}$ be a compact embedded minimal $\theta$-capillary hypersurface supported on $S_\rho$. If $S_\rho$ is a half-catenoid, then every component of $\Sigma$ is a horizontal slice. If $0<\theta\leq\pi/2$, no such hypersurface exists.
	\end{corollary}
	
	The next remark explains the roles of the two conditions in Theorem~\ref{thm:intro-disk-rigidity} and the counterexamples obtained when they are removed.
	
	\begin{remark}
		Theorem~\ref{thm:slope-condition-sharpness} shows that \eqref{eq:slope-condition} cannot in general be omitted from any part of
		Theorem~\ref{thm:intro-disk-rigidity}. For every $0<\theta<\pi$, it constructs a rotational support on which \eqref{eq:slope-condition} fails and which carries a non-horizontal planar $\theta$-capillary ball homologous to $D_0$. This gives a counterexample to part~(i), and for $\theta=\pi/2$ also to part~(ii). When $\theta\neq\pi/2$, the support can also be chosen to satisfy \eqref{eq:two-point-barrier}, giving a counterexample to part~(iii).
		
		Condition \eqref{eq:two-point-barrier} rules out catenoid bands and forces every component of $\Sigma$ to be homologous to $D_0$; see Proposition~\ref{prop:automatic-separation}. The boundary heights $a<b$ of a catenoid band satisfy
		\[
		A_\theta(b)=B_\theta(a),
		\]
		which is excluded by \eqref{eq:two-point-barrier}. For every $\pi/2<\theta<\pi$, Proposition~\ref{prop:profile-equality-nonrigid} constructs a support satisfying \eqref{eq:slope-condition} and carrying a catenoid band for which this equality holds. Hence \eqref{eq:two-point-barrier} cannot in general be omitted from part~(iii).
	\end{remark}
	
	For obtuse contact angles, the reverse profile inequality leads instead to catenoid-band rigidity. Assume
	\begin{equation}
		A_\theta(b)\geq B_\theta(a)\qquad(0<a<b<L).
		\label{eq:main-nonnegative-profile}
	\end{equation}
	We exclude cylindrical intervals by requiring
	\begin{equation}
		\rho'\not\equiv0\qquad\text{on any interval in }(0,\ell_\theta).
		\label{eq:no-cylinder-below-ell}
	\end{equation}
	Under these assumptions, we obtain the following theorem.
	
	\begin{theorem}
		\label{thm:intro-band-rigidity}
		Let $\pi/2<\theta<\pi$, and let $\Sigma\subset\overline{M_\rho}$ be a compact embedded minimal $\theta$-capillary hypersurface supported on $S_\rho$. Assume \eqref{eq:main-nonnegative-profile} and \eqref{eq:no-cylinder-below-ell}. Then every component of $\Sigma$ is a catenoid band.
	\end{theorem}
	
	Because the support has less symmetry, neither Alexandrov's moving-plane method nor Ros's integral method applies directly.
	The height function is harmonic on every minimal component of $\Sigma$. The Hopf boundary lemma and the contact-angle condition convert its boundary extrema into inequalities for $\rho'$, and \eqref{eq:slope-condition} then forces the height to be constant when the wetted region contains $D_0$. For a component $\Sigma_0$ of $\Sigma$ whose wetted region lies in $S_\rho$, integrating the height equation gives an identity on the wetted region. A vector field on $S_\rho$ depending only on the height allows the profile inequalities to determine the signs in this identity. The strict inequality \eqref{eq:two-point-barrier} gives a contradiction. Under the reverse inequality \eqref{eq:main-nonnegative-profile}, the identity instead forces every component of $\partial\Sigma_0$ to be a horizontal sphere. Unique continuation then makes $\Sigma_0$ rotational, and the minimality equation identifies it as a catenoid band.
	
	The paper is organized as follows. Section~\ref{sec:height-rigidity} proves
	Theorem~\ref{thm:intro-disk-rigidity} and
	Corollary~\ref{cor:half-catenoid}, and constructs counterexamples showing
	that the conclusion may fail without \eqref{eq:slope-condition}. Section~\ref{sec:obtuse-classification} proves Theorem~\ref{thm:intro-band-rigidity}. Appendix~\ref{app:equality-model} constructs a catenoid band showing that \eqref{eq:two-point-barrier} cannot in general be omitted. Appendix~\ref{app:angle-uniform} characterizes the support profiles for which \eqref{eq:two-point-barrier} holds for every obtuse contact angle.

  \medskip
\subsection*{Acknowledgments}
The authors thank Thomas Koerber for his interest in this work.

 \medskip
\noindent\textbf{Disclosure on AI assistance.}
The authors used AI-assisted tools, principally ChatGPT 5.6. The authors verified and completed all mathematical arguments, and take full responsibility for its content.
	
	\section{Alexandrov-type rigidity}
	\label{sec:height-rigidity}
	
	This section proves Theorem~\ref{thm:intro-disk-rigidity} and
	Corollary~\ref{cor:half-catenoid}, and shows that
	\eqref{eq:slope-condition} cannot in general be omitted.  The proof combines the height argument, an integral identity on the wetted region, and \eqref{eq:two-point-barrier}.
	
	Set $\widehat S_\rho=D_0\cup S_\rho$, and let $\Sigma\subset\overline{M_\rho}$ be a connected compact embedded minimal $\theta$-capillary hypersurface supported on $S_\rho$. Identify $\widehat S_\rho$ with $\R^n$. Each component of $\partial\Sigma$ separates $\widehat S_\rho$ into a bounded and an unbounded region. The closure of the points lying in an odd number of these bounded regions is the unique compact regular region $W_\Sigma\subset\widehat S_\rho$ with boundary $\partial\Sigma$. Likewise, $\Sigma\cup W_\Sigma$ separates $M_\rho$, and its bounded side is denoted by $\Omega_\Sigma$.
	
	Since $D_0$ is connected and disjoint from $\partial\Sigma$, either $D_0\subset W_\Sigma$ or $W_\Sigma\subset S_\rho$. Moreover, over $\mathbb Z_2$,
	\[
	\partial[\Omega_\Sigma]=[\Sigma]+[W_\Sigma].
	\]
	Thus
	\[
	D_0\subset W_\Sigma\quad\Longleftrightarrow\quad[\Sigma]=[D_0],
	\qquad
	W_\Sigma\subset S_\rho\quad\Longleftrightarrow\quad[\Sigma]=0.
	\]
	No orientation choice enters these homology classes because all coefficients are in $\mathbb Z_2$.
	
	Let $\mu_\Sigma$ be the unit conormal of $\partial\Sigma$ in $\Sigma$ pointing away from the interior of $\Sigma$. Let $\mu_{W_\Sigma}$ be the unit conormal of $\partial W_\Sigma$ in $W_\Sigma$ pointing from $W_\Sigma$ into $\widehat S_\rho\setminus W_\Sigma$, and set
	$\mu_{S_\rho}:=-\mu_{W_\Sigma}$. 
	Thus $\mu_{S_\rho}$ points from $\partial\Sigma$ into $W_\Sigma$ along $\widehat S_\rho$. When $\theta=\pi/2$, choose the sign of $\nu_\Sigma$ so that $\nu_\Sigma=\mu_{S_\rho}$ along $\partial\Sigma$. With these choices,
	\begin{equation}
		\begin{aligned}
			\mu_\Sigma&=\sin\theta\,N_{S_\rho}+\cos\theta\,\mu_{S_\rho},\\
			\nu_\Sigma&=-\cos\theta\,N_{S_\rho}+\sin\theta\,\mu_{S_\rho}.
		\end{aligned}
		\label{eq:normal-conormal-relations}
	\end{equation}
	
	Figure~\ref{fig:boundary-vectors} illustrates these orientation conventions in the normal plane to $\partial\Sigma$.
	
	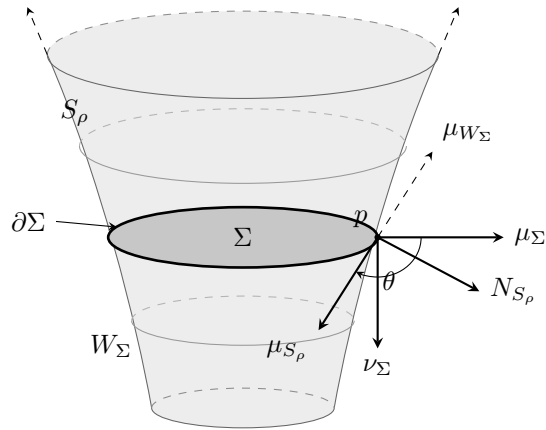
\begin{figure}[ht]
		\centering
		\begin{tikzpicture}[x=1.15cm,y=1.05cm,>=stealth,font=\small]
			\fill[black!7]
			(-1.05,0)
			.. controls (-1.22,1.05) and (-1.65,2.95) .. (-2.25,4.45)
			arc[start angle=180,end angle=0,x radius=2.25,y radius=0.55]
			.. controls (1.65,2.95) and (1.22,1.05) .. (1.05,0)
			arc[start angle=0,end angle=-180,x radius=1.05,y radius=0.25]--cycle;
			\draw[black!60]
			(-1.05,0).. controls (-1.22,1.05) and (-1.65,2.95) .. (-2.25,4.45);
			\draw[black!60]
			(1.05,0).. controls (1.22,1.05) and (1.65,2.95) .. (2.25,4.45);
			\draw[dashed,black!45] (-2.25,4.45)
			arc[start angle=180,end angle=0,x radius=2.25,y radius=0.55];
			\draw[black!60] (-2.25,4.45)
			arc[start angle=180,end angle=360,x radius=2.25,y radius=0.55];
			\draw[dashed,black!45] (-1.05,0)
			arc[start angle=180,end angle=0,x radius=1.05,y radius=0.25];
			\draw[black!60] (-1.05,0)
			arc[start angle=180,end angle=360,x radius=1.05,y radius=0.25];
			\draw[black!30,dashed] (-1.28,1.1)
			arc[start angle=180,end angle=0,x radius=1.28,y radius=0.31];
			\draw[black!40] (-1.28,1.1)
			arc[start angle=180,end angle=360,x radius=1.28,y radius=0.31];
			\draw[black!30,dashed] (-1.88,3.3)
			arc[start angle=180,end angle=0,x radius=1.88,y radius=0.47];
			\draw[black!40] (-1.88,3.3)
			arc[start angle=180,end angle=360,x radius=1.88,y radius=0.47];
			\fill[black!22] (0,2.15) ellipse[x radius=1.55,y radius=0.38];
			\draw[line width=1pt] (0,2.15) ellipse[x radius=1.55,y radius=0.38];
			\node at (0,2.15) {$\Sigma$};
			\node[left] at (-1.65,3.72) {$S_\rho$};
			\node[left] at (-1.16,0.78) {$W_\Sigma$};
			\draw[->] (-2.15,2.35) node[left] {$\partial\Sigma$}--(-1.45,2.28);
			\draw[dashed,->] (-2.25,4.45)--(-2.48,5.05);
			\draw[dashed,->] (2.25,4.45)--(2.48,5.05);
			\coordinate (p) at (1.55,2.15);
			\begin{scope}[shift={(p)}]
				\draw[->,line width=0.8pt] (0,0)--(1.45,0) node[right] {$\mu_\Sigma$};
				\draw[->,line width=0.8pt] (0,0)--(0,-1.4) node[below] {$\nu_\Sigma$};
				\draw[->,line width=0.8pt] (0,0)--(1.17,-0.68) node[right] {$N_{S_\rho}$};
				\draw[->,line width=0.8pt] (0,0)--(-0.68,-1.17) node[below left] {$\mu_{S_\rho}$};
				\draw[->,dashed] (0,0)--(0.63,1.08) node[above right] {$\mu_{W_\Sigma}$};
				\draw[->] (0.5,0) arc[start angle=0,end angle=-120,radius=0.5];
				\node at (0.13,-0.55) {$\theta$};
				\fill (0,0) circle (1.2pt) node[above left] {$p$};
			\end{scope}
		\end{tikzpicture}  
		\caption{Orientation conventions on the rotational support, illustrated for a horizontal slice $\Sigma$ meeting the support at an obtuse contact angle. $\Sigma$ meets $S_\rho$ along $\partial\Sigma$. At $p\in\partial\Sigma$, $\mu_{W_\Sigma}=-\mu_{S_\rho}$, and the vectors satisfy \eqref{eq:normal-conormal-relations}. The dashed upper edges indicate that $S_\rho$ continues upward.}
		\label{fig:boundary-vectors}
	\end{figure}
	
	For $h\in(0,L)$, write
	\[
	D_h:=\{(y,h):|y|\leq\rho(h)\}.
	\]
	
	We first apply the height function when the wetted region $W_\Sigma$ contains $D_0$.
	
	\begin{proposition}
		\label{prop:bottom-disk-rigidity}
		Let $0<\theta<\pi$, and let $\Sigma\subset\overline{M_\rho}$ be a connected compact embedded minimal $\theta$-capillary hypersurface supported on $S_\rho$, with $D_0\subset W_\Sigma$. Assume \eqref{eq:slope-condition}. Then $\Sigma$ is a horizontal slice.
	\end{proposition}
	
	\begin{proof}
		Suppose that $z|_\Sigma$ is nonconstant. Since $\Delta_\Sigma z=0$, the maximum principle gives points $p_\pm\in\partial\Sigma$ such that
		\[
		h_-:=z(p_-)=\min_\Sigma z<h_+:=\max_\Sigma z=z(p_+).
		\]
		Since $p_\pm$ are extrema of $z|_{\partial\Sigma}$, every $X\in T_{p_\pm}\partial\Sigma$ satisfies $X(z)=0$. Thus $T_{p_\pm}\partial\Sigma$ is contained in $T_{p_\pm}(S_\rho\cap\{z=h_\pm\})$. Both spaces have dimension $n-1$, and hence
		\[
		T_{p_\pm}\partial\Sigma=T_{p_\pm}\bigl(S_\rho\cap\{z=h_\pm\}\bigr).
		\]
		Thus $\mu_{S_\rho}(p_\pm)$ is parallel to $\nabla^{S_\rho}z(p_\pm)$. We next determine its direction. Since $\partial W_\Sigma=\partial\Sigma$ lies between heights $h_-$ and $h_+$,
		\[
		D_0\cup\bigl(S_\rho\cap\{0<z<h_-\}\bigr)\subset W_\Sigma,\qquad
		S_\rho\cap\{z>h_+\}\subset S_\rho\setminus W_\Sigma.
		\]
		The first inclusion follows from $D_0\subset W_\Sigma$, and the second from the compactness of $W_\Sigma$. At $p_-$, the first inclusion places $W_\Sigma$ on the side of decreasing height. At $p_+$, the second inclusion places $S_\rho\setminus W_\Sigma$ on the side of increasing height. Hence $\mu_{S_\rho}=-\mu_{W_\Sigma}$ points toward decreasing height at both points. At $(\rho(z)\omega,z)\in S_\rho$, where $\omega\in\Sph^{n-1}$ is the angular variable,
		\[
		\nabla^{S_\rho}z=\frac{(\rho'\omega,1)}{1+\rho'^2},
		\qquad
		N_{S_\rho}=\frac{(\omega,-\rho')}{\sqrt{1+\rho'^2}}.
		\]
		Therefore, at $p_\pm$,
		\[
		\mu_{S_\rho}
		=-\frac{\nabla^{S_\rho}z}{|\nabla^{S_\rho}z|}
		=\frac{(-\rho'\omega,-1)}{\sqrt{1+\rho'^2}}.
		\]
		Hence \eqref{eq:normal-conormal-relations} gives
		\begin{align*}
			\partial_{\mu_\Sigma} z
			&=\langle\mu_\Sigma,e_{n+1}\rangle\\
			&=\sin\theta\,\langle N_{S_\rho},e_{n+1}\rangle+\cos\theta\,\langle\mu_{S_\rho},e_{n+1}\rangle\\
			&=\sin\theta\left(-\frac{\rho'}{\sqrt{1+\rho'^2}}\right)-\cos\theta\left(\frac{1}{\sqrt{1+\rho'^2}}\right)\\
			&=\frac{-\rho'\sin\theta-\cos\theta}{\sqrt{1+\rho'^2}}.
		\end{align*}
		The Hopf boundary lemma now yields
		\[
		\frac{-\rho'(h_+)\sin\theta-\cos\theta}{\sqrt{1+\rho'(h_+)^2}}>0,\qquad
		\frac{-\rho'(h_-)\sin\theta-\cos\theta}{\sqrt{1+\rho'(h_-)^2}}<0.
		\]
		This implies
		\[
		\rho'(h_+)<-\cot\theta<\rho'(h_-),\qquad h_-<h_+,
		\]
		contradicting \eqref{eq:slope-condition}. Thus $z|_\Sigma\equiv h$. Each component of $\partial\Sigma$ is open and closed in the horizontal sphere $S_\rho\cap\{z=h\}$, so $\partial\Sigma=\partial D_h$. Since $\Sigma\subset D_h$ is connected and embedded, $\Sigma=D_h$. The contact angle condition then gives $\rho'(h)=-\cot\theta$.
	\end{proof}
	
	We now prepare the proof of Theorem~\ref{thm:intro-disk-rigidity}(iii). For $0<\theta<\pi$, $\theta\neq\pi/2$, Proposition~\ref{prop:integral-identity} gives an integral identity for any possible component of $\Sigma$ representing the zero relative homology class. Proposition~\ref{prop:automatic-separation} combines this identity with \eqref{eq:two-point-barrier} to exclude that class.
	
	The first step is the integral identity for the case $W_\Sigma\subset S_\rho$.
	
	\begin{proposition} 
		\label{prop:integral-identity}
		Let $0<\theta<\pi$ with $\theta\neq\pi/2$, and let $\Sigma\subset\overline{M_\rho}$ be a connected compact embedded minimal $\theta$-capillary hypersurface supported on $S_\rho$. If $[\Sigma]=0$, then
		\[
		\int_{W_\Sigma}\Delta_{S_\rho}z\,dA
		+\tan\theta\int_{\partial W_\Sigma}
		\frac{\rho'}{\sqrt{1+\rho'^2}}\,ds=0.
		\]
	\end{proposition}
	
	\begin{proof}
		The condition $[\Sigma]=0$ gives $W_\Sigma\subset S_\rho$ and $\partial W_\Sigma=\partial\Sigma$. Since $\Delta_\Sigma z=0$, \eqref{eq:normal-conormal-relations} and
		\[
		\langle N_{S_\rho},e_{n+1}\rangle=-\frac{\rho'}{\sqrt{1+\rho'^2}},
		\qquad
		\langle\mu_{S_\rho},e_{n+1}\rangle
		=-\langle\mu_{W_\Sigma},\nabla^{S_\rho}z\rangle
		\]
		give
		\begin{align*}
			0
			&=\int_{\partial\Sigma}\langle\mu_\Sigma,e_{n+1}\rangle\,ds\\
			&=-\sin\theta\int_{\partial W_\Sigma}
			\frac{\rho'}{\sqrt{1+\rho'^2}}\,ds
			-\cos\theta\int_{\partial W_\Sigma}
			\langle\mu_{W_\Sigma},\nabla^{S_\rho}z\rangle\,ds\\
			&=-\sin\theta\int_{\partial W_\Sigma}
			\frac{\rho'}{\sqrt{1+\rho'^2}}\,ds
			-\cos\theta\int_{W_\Sigma}\Delta_{S_\rho}z\,dA\\
			&=-\cos\theta\left(
			\int_{W_\Sigma}\Delta_{S_\rho}z\,dA
			+\tan\theta\int_{\partial W_\Sigma}
			\frac{\rho'}{\sqrt{1+\rho'^2}}\,ds\right).
		\end{align*}
		Since $\cos\theta\neq0$, the desired identity follows.
	\end{proof}
	
	For the comparison argument, set
	$$J=\frac{\rho^{n-1}}{\sqrt{1+\rho'^2}},$$ 
	so $A_\theta=J(1+\tan\theta\,\rho')$ and $B_\theta=J(1-\tan\theta\,\rho')$.
	
	The strict profile inequality then excludes $W_\Sigma\subset S_\rho$.
	
	\begin{proposition}
		\label{prop:automatic-separation}
		Let $0<\theta<\pi$ with $\theta\neq\pi/2$, and let $\Sigma\subset\overline{M_\rho}$ be a connected compact embedded minimal $\theta$-capillary hypersurface supported on $S_\rho$. If \eqref{eq:two-point-barrier} holds, then $[\Sigma]=[D_0]$.
	\end{proposition}
	
	\begin{proof} 
		For $0<a<\ell_\theta$, letting $b\downarrow a$ in \eqref{eq:two-point-barrier} gives
		\[
		0\leq B_\theta(a)-A_\theta(a)
		=-2J(a)\tan\theta\,\rho'(a).
		\]
		Thus
		\[
		\tan\theta\,\rho'\leq0
		\qquad\text{on }(0,\ell_\theta).
		\]
		Suppose $[\Sigma]=0$. The height function $z$ on $\Sigma$ is nonconstant, and the Hopf lemma at a highest boundary point gives $\rho'(h_+)<-\cot\theta$. Continuity gives $h_+<\ell_\theta$. Moreover, $|\nabla^{S_\rho}z|=(1+\rho'^2)^{-1/2}>0$, so the extrema of $z$ on each component $E$ of $W_\Sigma$ occur on $\partial E\subset\partial\Sigma$. Hence
		\[
		0<\min_Ez\leq\max_Ez\leq h_+<\ell_\theta.
		\]
		Set
		\[
		\alpha=\min_E z,\qquad \beta=\max_E z,\qquad
		K_-(z)=\max_{z\leq s\leq\beta}A_\theta(s).
		\]
		Then $K_-$ is Lipschitz and nonincreasing, and $K_-(z)\geq A_\theta(z)$. 
		Since $J>0$ and $\tan\theta\,\rho'\leq0$ on $(0,\ell_\theta)$,
		\[
		B_\theta(z)-A_\theta(z)=-2J(z)\tan\theta\,\rho'(z)\geq0.
		\]
		For $s>z$, \eqref{eq:two-point-barrier} gives $A_\theta(s)<B_\theta(z)$, while $A_\theta(z)\leq B_\theta(z)$. Hence
		\[
		A_\theta(z)\leq K_-(z)\leq B_\theta(z)
		\qquad(\alpha\leq z\leq\beta).
		\]
		Since $J>0$,
		\[
		\tan\theta\,\rho'
		\leq\frac{K_-}{J}-1
		\leq-\tan\theta\,\rho',
		\qquad
		\left|\frac{1-K_-/J}{-\tan\theta}\right|\leq|\rho'|.
		\]
		
		Define
		\[
		X_-=\frac{1-K_-/J}{-\tan\theta}\nabla^{S_\rho}z.
		\]
		Since $|\nabla^{S_\rho}z|=(1+\rho'^2)^{-1/2}$,  
		\[
		|X_-|\leq\frac{|\rho'|}{\sqrt{1+\rho'^2}}.
		\]
		Let $\mu_E$ be the outward unit conormal of $\partial E$ in $E$. The Cauchy--Schwarz inequality and $\tan\theta\,\rho'\leq0$ give
		\begin{equation}
			\begin{aligned}
				\tan\theta\,\langle X_-,\mu_E\rangle
				&\geq-|\tan\theta|\,|X_-|  \geq-|\tan\theta|\frac{|\rho'|}{\sqrt{1+\rho'^2}}
				=\tan\theta\frac{\rho'}{\sqrt{1+\rho'^2}}
				\qquad\text{on }\partial E.
			\end{aligned}
			\label{eq:component-boundary-control}
		\end{equation}
		
		We next compute $\operatorname{div}_{S_\rho}X_-$. The induced metric and the gradient of the height function on $S_\rho$ are
		\[
		g_{S_\rho}=(1+\rho'^2)\,dz^2+\rho^2g_{\Sph^{n-1}},
		\qquad
		\nabla^{S_\rho}z=\frac{1}{1+\rho'^2}\,\partial_z.
		\]
		Therefore, for every function $\phi=\phi(z)$, the coordinate formula for divergence gives
		\begin{align*}
			\operatorname{div}_{S_\rho}(\phi\nabla^{S_\rho}z)
			&=\frac1{\rho^{n-1}\sqrt{1+\rho'^2}}
			\frac{d}{dz}\left(
			\rho^{n-1}\sqrt{1+\rho'^2}\,\frac{\phi}{1+\rho'^2}
			\right)\\
			&=\frac1{\rho^{n-1}\sqrt{1+\rho'^2}}
			\frac{d}{dz}\left(\frac{\rho^{n-1}\phi}{\sqrt{1+\rho'^2}}\right).
		\end{align*}
		Taking $\phi=1$ and then $\phi=(1-K_-/J)/(-\tan\theta)$ gives
		\begin{align*}
			\Delta_{S_\rho}z
			&=\frac1{\rho^{n-1}\sqrt{1+\rho'^2}}\frac{dJ}{dz}
			=\frac{J'}{\rho^{n-1}\sqrt{1+\rho'^2}},\\
			\operatorname{div}_{S_\rho}X_-
			&=\frac1{\rho^{n-1}\sqrt{1+\rho'^2}}
			\frac{d}{dz}\left(\frac{K_--J}{\tan\theta}\right)\\
			&=\frac{(K_--J)'}{\tan\theta\,\rho^{n-1}\sqrt{1+\rho'^2}}.
		\end{align*}
		Since $K_-'\leq0$ almost everywhere,
		\begin{equation}
			\Delta_{S_\rho}z+\tan\theta\,\operatorname{div}_{S_\rho}X_-
			=\frac{K_-'}{\rho^{n-1}\sqrt{1+\rho'^2}}\leq0.
			\label{eq:component-interior-control}
		\end{equation}
		Integrating \eqref{eq:component-interior-control} over $E$ and applying the divergence theorem together with \eqref{eq:component-boundary-control}, we obtain
		\begin{equation}
			\begin{aligned}
				&\int_E\Delta_{S_\rho}z\,dA
				+\tan\theta\int_{\partial E}\frac{\rho'}{\sqrt{1+\rho'^2}}\,ds\\
				&=\int_E\left(\Delta_{S_\rho}z
				+\tan\theta\,\operatorname{div}_{S_\rho}X_-\right)dA\\
				&\quad+\tan\theta\int_{\partial E}
				\left(\frac{\rho'}{\sqrt{1+\rho'^2}}-\langle X_-,\mu_E\rangle\right)ds
				\leq0.
			\end{aligned}
			\label{eq:component-integral-comparison}
		\end{equation}
		We show that the last inequality is strict. If $\rho'(\alpha)\neq0$, choose $p\in\partial E$ with $z(p)=\alpha$. Then
		\[
		\mu_E(p)=-\frac{\nabla^{S_\rho}z}{|\nabla^{S_\rho}z|}(p).
		\]
		Since $\tan\theta\,\rho'(\alpha)<0$,
		\[
		A_\theta(\alpha)<B_\theta(\alpha).
		\]
		For $\alpha<s\leq\beta$, \eqref{eq:two-point-barrier} gives
		\[
		A_\theta(s)<B_\theta(\alpha).
		\]
		Therefore 
		\[
		K_-(\alpha)=\max_{\alpha\leq s\leq\beta}A_\theta(s)
		<B_\theta(\alpha).
		\]
		Using the definition of $X_-$ and
		$|\nabla^{S_\rho}z|=(1+\rho'^2)^{-1/2}$, we obtain
		\begin{align*}
			&\tan\theta\left(
			\frac{\rho'(\alpha)}{\sqrt{1+\rho'(\alpha)^2}}
			-\langle X_-,\mu_E\rangle(p)\right)
			\\=&\tan\theta\left(
			\frac{\rho'(\alpha)}{\sqrt{1+\rho'(\alpha)^2}}
			+\frac{1-K_-(\alpha)/J(\alpha)}{-\tan\theta}
			\frac{1}{\sqrt{1+\rho'(\alpha)^2}}\right)\\
			=&\frac{K_-(\alpha)/J(\alpha)
				-\bigl(1-\tan\theta\,\rho'(\alpha)\bigr)}
			{\sqrt{1+\rho'(\alpha)^2}} 
			= \frac{K_-(\alpha)-B_\theta(\alpha)}
			{J(\alpha)\sqrt{1+\rho'(\alpha)^2}}<0.
		\end{align*}
		Thus the boundary integrand in \eqref{eq:component-integral-comparison} is negative at $p$. It is nonpositive everywhere by \eqref{eq:component-boundary-control}, so its integral is strictly negative. Together with \eqref{eq:component-interior-control}, this makes the right-hand side of \eqref{eq:component-integral-comparison} strictly negative.
		
		It remains to consider $\rho'(\alpha)=0$. For $\alpha<s\leq\beta$, \eqref{eq:two-point-barrier} gives
		\[
		A_\theta(s)<B_\theta(\alpha)=A_\theta(\alpha)
		\]
		and hence
		\[
		K_-(\alpha)=A_\theta(\alpha)>A_\theta(\beta)
		=K_-(\beta).
		\]
		
		Since $K_-$ is Lipschitz and nonincreasing, $K_-'<0$ on a set of positive measure. The interior of $E$ is connected and $|\nabla^{S_\rho}z|>0$, so for every $t\in(\alpha,\beta)$, the slice $E\cap\{z=t\}$ contains a relatively open subset of the horizontal sphere and hence has positive $(n-1)$-dimensional measure. By the coarea formula and $|\nabla^{S_\rho}z|=(1+\rho'^2)^{-1/2}$,
		\[
		\int_E\frac{K_-'(z)}{\rho^{n-1}\sqrt{1+\rho'^2}}\,dA
		=\int_\alpha^\beta\frac{K_-'(t)}{\rho(t)^{n-1}}
		\mathcal H^{n-1}\bigl(E\cap\{z=t\}\bigr)\,dt<0.
		\]
		Hence the integral of the right-hand side of \eqref{eq:component-interior-control} over $E$ is strictly negative, while the boundary term in \eqref{eq:component-integral-comparison} is nonpositive.
		
		Therefore, in either case, the left-hand side of \eqref{eq:component-integral-comparison} is strictly negative for every component $E$ of $W_\Sigma$. Summing over $E$ contradicts Proposition~\ref{prop:integral-identity}. Hence $[\Sigma]=[D_0]$.
	\end{proof}
	
	We now prove Theorem~\ref{thm:intro-disk-rigidity} and Corollary~\ref{cor:half-catenoid}.
	
	\begin{proof}[Proof of Theorem~\ref{thm:intro-disk-rigidity}]
		For part~(i), let $\Sigma_0$ be a component of $\Sigma$ with $[\Sigma_0]=[D_0]$. Then $D_0\subset W_{\Sigma_0}$, and Proposition~\ref{prop:bottom-disk-rigidity} implies that $\Sigma_0$ is a horizontal slice.
		
		For part~(ii), suppose that $\theta=\pi/2$, fix a component $\Sigma_0$ of $\Sigma$, and suppose that $z|_{\Sigma_0}$ is nonconstant. Since $z|_{\Sigma_0}$ is harmonic, its extrema are attained at points $p_\pm\in\partial\Sigma_0$, with
		\[
		z(p_-)=\min_{\Sigma_0}z=:h_-<h_+:=\max_{\Sigma_0}z=z(p_+).
		\]
		The Hopf boundary lemma reduces
		\[
		\partial_{\mu_{\Sigma_0}}z(p_+)>0,
		\qquad
		\partial_{\mu_{\Sigma_0}}z(p_-)<0.
		\]
		Since $\theta=\pi/2$, \eqref{eq:normal-conormal-relations} gives
		\[
		\partial_{\mu_{\Sigma_0}}z=-\frac{\rho'}{\sqrt{1+\rho'^2}}
		\qquad\text{on }\partial\Sigma_0.
		\]
		Consequently,
		\[
		\rho'(h_+)<0<\rho'(h_-),
		\qquad h_-<h_+,
		\]
		contradicting \eqref{eq:slope-condition}. Hence $z|_{\Sigma_0}\equiv h$. As in Proposition~\ref{prop:bottom-disk-rigidity}, connectedness gives $\Sigma_0=D_h$. Since $\Sigma_0$ is arbitrary, every component of $\Sigma$ is a horizontal slice.
		
		For part~(iii), fix a component $\Sigma_0$ of $\Sigma$. Proposition~\ref{prop:automatic-separation} gives $[\Sigma_0]=[D_0]$. Part~(i) then gives that $\Sigma_0$ is a horizontal slice.
	\end{proof}

	\begin{proof}[Proof of Corollary~\ref{cor:half-catenoid}]
		For the half-catenoid,
		\[
		\rho\rho''=(n-1)(1+\rho'^2),
		\]
		hence $\rho''>0$, $\rho'>0$ on $(0,L)$, and \eqref{eq:slope-condition} holds for every $0<\theta<\pi$.
		
		We first prove that every component of $\Sigma$ is a horizontal slice. For $\theta=\pi/2$, this follows from Theorem~\ref{thm:intro-disk-rigidity}(ii). Suppose $\theta\neq\pi/2$, and fix a component $\Sigma_0$ of $\Sigma$. Since the half-catenoid is minimal,
		\[
		\Delta_{S_\rho}z=0.
		\]
		If $[\Sigma_0]=0$, Proposition~\ref{prop:integral-identity} gives
		\[
		\tan\theta\int_{\partial W_{\Sigma_0}}
		\frac{\rho'}{\sqrt{1+\rho'^2}}\,ds=0.
		\]
		This is impossible because $\tan\theta\neq0$, $\rho'>0$ on $(0,L)$, and $\partial W_{\Sigma_0}=\partial\Sigma_0$ is nonempty. Hence $[\Sigma_0]=[D_0]$, and Proposition~\ref{prop:bottom-disk-rigidity} gives that $\Sigma_0$ is a horizontal slice. Since $\Sigma_0$ is arbitrary, every component of $\Sigma$ is a horizontal slice.
		
		It remains to prove the nonexistence statement. The condition $\Sigma\setminus\partial\Sigma\subset\operatorname{int}M_\rho$ excludes the horizon $D_0$, so every horizontal slice under consideration is $D_h$ for some $h>0$. If $0<\theta\leq\pi/2$, such a slice would satisfy
		\[
		\rho'(h)=-\cot\theta\leq0,
		\]
		contrary to $\rho'(h)>0$. Thus no such hypersurface exists for $0<\theta\leq\pi/2$.
	\end{proof} 
	
	Having proved Theorem~\ref{thm:intro-disk-rigidity} and Corollary~\ref{cor:half-catenoid}, we now show that the slope condition \eqref{eq:slope-condition} cannot be omitted from Theorem~\ref{thm:intro-disk-rigidity}. For every $0<\theta<\pi$, the conclusion may fail for a component of $\Sigma$ homologous to $D_0$. When $\theta\neq\pi/2$, the strict profile inequality \eqref{eq:two-point-barrier} may still hold.

	\begin{theorem}
		\label{thm:slope-condition-sharpness}
		Fix $n\geq2$ and $0<\theta<\pi$. There exist a smooth one-ended rotational support $S_\rho$ without cylindrical intervals and a compact embedded minimal $\theta$-capillary hypersurface $\Sigma$ supported on $S_\rho$ such that $D_0\subset W_\Sigma$ and \eqref{eq:slope-condition} fails, while $\Sigma$ is a nonhorizontal planar $n$-ball. If $\theta\neq\pi/2$, the support can also be chosen so that \eqref{eq:two-point-barrier} holds. Thus the conclusion of Theorem~\ref{thm:intro-disk-rigidity} fails if \eqref{eq:slope-condition} is omitted.
	\end{theorem}

	\begin{proof}
		The construction is obtained by cutting a sphere centered on the $z$-axis with a slightly tilted hyperplane. The sphere is rotationally symmetric, whereas the resulting planar $n$-ball is not horizontal. We then extend the spherical profile above the $n$-ball without changing the $n$-ball.
		
		\smallskip
		\textit{Step 1.} We construct a spherical profile that violates \eqref{eq:slope-condition}.
		Choose $\delta>0$ such that $0<\theta-\delta<\theta+\delta<\pi$, and set
		\[
		H_0=\cos(\theta-\delta)-\cos(\theta+\delta),
		\qquad
		\rho_0(z)=\sqrt{1-\bigl(z+\cos(\theta+\delta)\bigr)^2}
		\quad(0\leq z\leq H_0).
		\]
		Then
		\begin{align*}
			\rho_0'(z)
			&=-\frac{z+\cos(\theta+\delta)}{\rho_0(z)},\\
			\rho_0''(z)
			&=-\frac{\rho_0(z)^2+\bigl(z+\cos(\theta+\delta)\bigr)^2}{\rho_0(z)^3}
			=-\frac1{\rho_0(z)^3}<0.
		\end{align*}
		Moreover,
		\[
		\rho_0'(0)=-\cot(\theta+\delta)>-\cot\theta>-\cot(\theta-\delta)=\rho_0'(H_0).
		\]
		Since $\rho_0''<0$, the slope $\rho_0'$ is strictly decreasing. Hence, for every sufficiently small $a>0$,
		\[
		\rho_0'(a)>-\cot\theta>\rho_0'(H_0).
		\]
		The pair $a<H_0$ therefore violates \eqref{eq:slope-condition}.
		
		\smallskip
		\textit{Step 2.} We construct a nonhorizontal planar $\theta$-capillary $n$-ball on the spherical part of the support.
		The rotation of the graph of $\rho_0$ is the part $0\leq z\leq H_0$ of the unit sphere centered at $-\cos(\theta+\delta)e_{n+1}$. It is therefore rotationally symmetric about the $z$-axis. Choose a unit vector $\eta$ making angle $\delta/2$ with $e_{n+1}$, and define
		\[
		\Sigma=\left\{x\in\overline{B_1(-\cos(\theta+\delta)e_{n+1})}:
		\left\langle x+\cos(\theta+\delta)e_{n+1},\eta\right\rangle=\cos\theta\right\},
		\qquad \nu_\Sigma=-\eta.
		\]
		Since $\eta\neq e_{n+1}$, the hyperplane containing $\Sigma$ is not horizontal. Thus $\Sigma$ is an embedded nonhorizontal minimal $n$-ball, and $\partial\Sigma$ is the intersection of this hyperplane with the unit sphere.
		
		We next verify that $\partial\Sigma$ lies in the part $0<z<H_0$ of the sphere. For $x\in\partial\Sigma$, write
		\[
		x+\cos(\theta+\delta)e_{n+1}
		=\cos\theta\,\eta+\sin\theta\,\xi,
		\qquad |\xi|=1,\qquad \langle\xi,\eta\rangle=0.
		\]
		Since $|\langle\xi,e_{n+1}\rangle|\leq\sin(\delta/2)$,
		\[
		0<\cos\left(\theta+\frac\delta2\right)-\cos(\theta+\delta)
		\leq z(x)
		\leq\cos\left(\theta-\frac\delta2\right)-\cos(\theta+\delta)
		<H_0.
		\]
		Because the height function is affine on $\Sigma$, its extrema occur on $\partial\Sigma$; hence the same bounds hold throughout $\Sigma$.
		
		The outward unit normal of the unit sphere is
		\[
		N_{S_{\rho_0}}=x+\cos(\theta+\delta)e_{n+1}.
		\]
		On $\partial\Sigma$, the defining equation of the hyperplane gives
		\[
		\langle\nu_\Sigma,N_{S_{\rho_0}}\rangle=-\cos\theta.
		\]
		Thus $\Sigma$ meets the rotational support at the constant angle $\theta$. Therefore $\Sigma$ is a nonhorizontal planar minimal $\theta$-capillary $n$-ball supported on the spherical part of $S_{\rho_0}$.

		\smallskip
		\textit{Step 3.} We verify that $D_0\subset W_\Sigma$.
		The hyperplane containing $\Sigma$ is the level set
		\[
		\left\langle x+\cos(\theta+\delta)e_{n+1},\eta\right\rangle=\cos\theta.
		\]
		Since $\rho_0(0)=\sin(\theta+\delta)$,
		\begin{align*}
			\max_{D_0}\left\langle x+\cos(\theta+\delta)e_{n+1},\eta\right\rangle
			&=\rho_0(0)\sin\frac\delta2
			+\cos(\theta+\delta)\cos\frac\delta2\\
			&=\sin(\theta+\delta)\sin\frac\delta2
			+\cos(\theta+\delta)\cos\frac\delta2\\
			&=\cos\left(\theta+\frac\delta2\right)<\cos\theta.
		\end{align*}
		Thus $D_0$ lies on the side of the hyperplane whose boundary inside the unit ball consists of $\Sigma$ and the spherical support. Hence $D_0\subset W_\Sigma$.

		\textit{Step 4.} We extend the spherical profile to a smooth one-ended support without changing $\Sigma$. Since $\Sigma\subset\{z<H_0\}$, it suffices to keep $\rho=\rho_0$ on $[0,H_0]$. If $\theta=\pi/2$, fix $C>|\rho_0'(H_0)|+1$ and choose
		\[
		0<\varepsilon<\frac{\rho_0(H_0)}{C}.
		\]
		Choose a constant $c\in(0,C)$ and a smooth function $p$ such that $p=\rho_0'$ near $H_0$,
		\[
		|p|\leq C\quad\text{on }[H_0,H_0+\varepsilon],
		\qquad p\equiv c\quad\text{on }[H_0+\varepsilon,\infty),
		\]
		and $p$ has a single zero in $(H_0,H_0+\varepsilon)$. Set
		\[
		\rho(z)=\rho_0(H_0)+\int_{H_0}^z p(t)\,dt
		\qquad(z\geq H_0).
		\]
		Then
		\[
		\rho(z)\geq\rho_0(H_0)-C\varepsilon>0
		\qquad(H_0\leq z\leq H_0+\varepsilon).
		\]
		For $z\geq H_0+\varepsilon$, one has $\rho'=c>0$, so $\rho$ remains positive. Hence $S_\rho$ is complete, proper, one-ended, and has no upper cap or cylindrical interval. This proves the case $\theta=\pi/2$.
		
		Assume now that $\theta\neq\pi/2$. At $(\rho,\rho')=(\sin\theta,-\cot\theta)$, \eqref{eq:profile-barriers} gives
		\[
		A_\theta=0<2\sin^n\theta=B_\theta.
		\]
		By continuity, there is a neighborhood $U$ of $(\sin\theta,-\cot\theta)$ such that
		\[
		A_\theta(v)<B_\theta(u)
		\]
		whenever the two pairs $(\rho(u),\rho'(u))$ and $(\rho(v),\rho'(v))$ belong to $U$. Since Steps~1--3 hold for every sufficiently small $\delta$, choose $\delta$ so that
		\[
		[\theta-\delta,\theta+\delta]\cap\{\pi/2\}=\varnothing,\qquad
		\{(\sin t,-\cot t):\theta-\delta\leq t\leq\theta+\delta\}\subset U.
		\]
		Then
		\[
		\{(\rho_0(z),\rho_0'(z)):0\leq z\leq H_0\}\subset U.
		\]
		
		Choose $H>H_0$ and extend $\rho_0$ smoothly to a positive $\rho$ on $[0,\infty)$ such that $\rho=\rho_0$ on $[0,H_0]$,
		\[
		\rho'<-\cot\theta\quad\text{on }[H_0,H),\qquad
		\rho'(H)=-\cot\theta,\qquad
		\rho'>-\cot\theta\quad\text{on }(H,\infty),
		\]
		\[
		(\rho(z),\rho'(z))\in U\qquad(H_0\leq z\leq H),
		\]
		and $\rho'$ is a positive constant for all sufficiently large $z$. To construct the extension, choose a smooth function $p$ agreeing with $\rho_0'$ near $H_0$ such that
		\[
		p<-\cot\theta\quad\text{on }[H_0,H),\qquad p(H)=-\cot\theta,\qquad p>-\cot\theta\quad\text{on }(H,\infty),
		\]
		with $p\not\equiv0$ on any interval and $p$ equal to a positive constant outside a compact set. Take $H-H_0$ small, keep $p=\rho_0'$ near $H_0$, and arrange
		\[
		p(z)=-\cot\theta+\lambda(z-H)
		\]
		near $H$ for some $\lambda>0$. Choose $p<-\cot\theta$ on $[H_0,H)$ and $p>-\cot\theta$ on $(H,\infty)$, with $p$ equal to a positive constant for all sufficiently large $z$. Set
		\[
		\rho(z)=\rho_0(H_0)+\int_{H_0}^z p(t)\,dt.
		\]
		Taking $H-H_0$ sufficiently small and $p$ sufficiently close to $\rho_0'$ on $[H_0,H]$ ensures $\rho>0$ and $(\rho,\rho')\in U$ there. The resulting $S_\rho$ is complete, proper, one-ended, and has no upper cap or cylindrical interval. Moreover,
		\[
		\ell_\theta=H,\qquad (\rho(z),\rho'(z))\in U\quad(0\leq z\leq H).
		\]
		Hence, for every $0<u<v<\ell_\theta$,
		\[
		A_\theta(v)<B_\theta(u),
		\]
		so \eqref{eq:two-point-barrier} holds. Since $\rho=\rho_0$ on $[0,H_0]$, the pair $a<H_0$ from Step~1 still violates \eqref{eq:slope-condition}, and the $n$-ball constructed in Step~2 is unchanged.
	\end{proof}
	
	\section{Catenoid-band rigidity}
	\label{sec:obtuse-classification}
	
	This section describes the catenoid-band alternative ruled out by \eqref{eq:two-point-barrier}. We prove Theorem~\ref{thm:intro-band-rigidity}: under the reverse profile inequality \eqref{eq:main-nonnegative-profile} and condition \eqref{eq:no-cylinder-below-ell}, every component of $\Sigma$ is an admissible catenoid band.
	
	An \textit{admissible $\theta$-capillary catenoid band} is a connected rotational minimal hypersurface
	\[
	\Sigma_f=\{(f(z)\omega,z):a\leq z\leq b,\ \omega\in\Sph^{n-1}\},\qquad 0<a<b<\ell_\theta,
	\]
	whose boundary consists of the two horizontal spheres in $S_\rho$ at heights $a$ and $b$, whose interior lies in $\operatorname{int}M_\rho$, and which meets $S_\rho$ at the constant contact angle $\theta$ along both components of $\partial\Sigma_f$.
	
	Write an admissible catenoid band as $\Sigma_f$, where $f(z)>0$ is its radius at height $z$. Minimality, the requirement that the interior lie inside $S_\rho$, and the two contact-angle conditions are respectively
	\begin{align}
		ff''&=(n-1)(1+f'^2),
		\label{eq:band-catenoid-ode}\\
		f(a)&=\rho(a),\quad f(b)=\rho(b),\quad f<\rho\ \text{on }(a,b),
		\label{eq:band-inside}\\
		\arctan\rho'(a)-\arctan f'(a)&=\pi-\theta,\qquad \arctan\rho'(b)-\arctan f'(b)=-(\pi-\theta).
		\label{eq:band-angles}
	\end{align} 
	Differentiating and using the minimality equation, we obtain
	\begin{align*}
		\frac{d}{dz}\left(\frac{f^{n-1}}{\sqrt{1+f'^2}}\right)
		&=\frac{f^{n-2}f'}{(1+f'^2)^{3/2}}
		\bigl((n-1)(1+f'^2)-ff''\bigr)=0.
	\end{align*}
	Therefore
	\begin{equation}
		\frac{f^{n-1}}{\sqrt{1+f'^2}}=c^{n-1}
		\label{eq:catenoid-first-integral}
	\end{equation}
	for some constant $c>0$.
	
	The following lemma shows that a component in the zero relative homology class whose boundary components are horizontal spheres is a catenoid band.
	
	\begin{lemma}
		\label{lem:horizontal-boundary-catenoid}
		Let $0<\theta<\pi$, and let $\Sigma_0\subset\overline{M_\rho}$ be a connected compact embedded minimal $\theta$-capillary hypersurface supported on $S_\rho$. Assume $[\Sigma_0]=0$ and that every component of $\partial\Sigma_0$ is a horizontal sphere. Then $\Sigma_0$ is an admissible catenoid band. If $\theta\neq\pi/2$ and its boundary heights are $a<b$, then
		\[
		A_\theta(b)=B_\theta(a).
		\]
	\end{lemma}
	
	\begin{proof}
		\smallskip
		\textit{Step 1.}
		We prove that $\Sigma_0$ is rotationally symmetric.
		Let $R$ be a rotation about the vertical axis, and fix a component $\Lambda$ of $\partial\Sigma_0$. Since $\Lambda$ is a horizontal sphere, $R(\Lambda)=\Lambda$. Moreover, \eqref{eq:normal-conormal-relations} gives
		\[
		T_p\Sigma_0=T_p\Lambda\oplus\operatorname{span}\{\sin\theta\,N_{S_\rho}(p)+\cos\theta\,\mu_{S_\rho}(p)\},
		\qquad p\in\Lambda.
		\]
		Thus $R(\Sigma_0)$ and $\Sigma_0$ coincide along $\Lambda$ and have the same tangent spaces there. Near any $p\in\Lambda$, write them as minimal graphs $u$ and $v$ over the same side of their common tangent hyperplane. On their common projected domain, the difference $w=u-v$ satisfies
		\[
		\operatorname{div}(\mathcal A(x)\nabla w)=0,
		\]
		where
		\[
		\mathcal A_{ij}(x)=\int_0^1\left(
		\frac{\delta_{ij}}{\sqrt{1+|q_t|^2}}
		-\frac{(q_t)_i(q_t)_j}{(1+|q_t|^2)^{3/2}}
		\right)dt,
		\qquad q_t=\nabla v+t(\nabla u-\nabla v).
		\]
		The matrix $\mathcal A$ is uniformly elliptic and Lipschitz in a sufficiently small neighborhood. Flatten the projected image of $\Lambda$ by a smooth change of coordinates, and write
		\[
		B^\pm=B\cap\{\pm x_n>0\},
		\qquad \Gamma=B\cap\{x_n=0\}.
		\]
		Denoting the transformed coefficients again by $\mathcal A$, we have  
		\[
		\operatorname{div}(\mathcal A\nabla w)=0\quad\text{in }B^+,
		\qquad w=0,\quad\nabla w=0\quad\text{on }\Gamma.
		\]
		Extend $\mathcal A$ to a uniformly elliptic Lipschitz matrix $\widetilde{\mathcal A}$ on $B$, and define
		\[
		\widetilde w(x)=
		\begin{cases}
		w(x),&x\in B^+,\\
		0,&x\in B^-\cup\Gamma.
		\end{cases}
		\]
		For every $\phi\in C_c^\infty(B)$, if $\eta$ is the outward unit normal of $B^+$ along $\Gamma$, then
		\begin{align*}
		\int_B\langle\widetilde{\mathcal A}\nabla\widetilde w,\nabla\phi\rangle\,dx
		&=\int_{B^+}\langle\mathcal A\nabla w,\nabla\phi\rangle\,dx\\
		&=\int_\Gamma\phi\,\mathcal A\nabla w\cdot\eta\,dS
		-\int_{B^+}\phi\,\operatorname{div}(\mathcal A\nabla w)\,dx=0.
		\end{align*}
		Thus $\widetilde w$ is a weak solution on $B$. Since $\widetilde w=0$ on $B^-$. Applying unique continuation for elliptic equations \cite{Aronszajn}, we see $\widetilde w\equiv0$ on $B$. Hence $R(\Sigma_0)$ and $\Sigma_0$ agree near $\Lambda$. Interior unique continuation and connectedness then give $R(\Sigma_0)=\Sigma_0$ globally. Since $R$ is arbitrary, $\Sigma_0$ is rotationally symmetric.
		
		\smallskip
		\textit{Step 2.}
		We show that the height is strictly monotone along the generating curve.
		The generating curve is an interval; parametrize it by arclength as
		\[
		\gamma(s)=(r(s),z(s)),\qquad \gamma'(s)=(\cos\varphi(s),\sin\varphi(s)).
		\] 
		Choose $\nu_{\Sigma_0}=(-\sin\varphi\,\omega,\cos\varphi)$. The principal curvatures are $\varphi_s$ and $\sin\varphi/r$ with multiplicities $1$ and $n-1$, respectively. 
		Hence minimality gives
		\[
		\varphi_s+(n-1)\frac{\sin\varphi}{r}=0.
		\]
		Since $z_{ss}=\cos\varphi\,\varphi_s=r_s\varphi_s$, we have
		\[
		\frac{d}{ds}\left(r^{n-1}z_s\right)
		=r^{n-2}r_s\bigl((n-1)z_s+r\varphi_s\bigr)=0.
		\]
	Hence
	\[
	r^{n-1}z_s=C
	\]
	for some constant $C$. If $C=0$, then $z_s=0$ wherever $r>0$. Since the curve is parametrized by arclength, $|r_s|=1$, so the generating curve is a horizontal radial segment and $\Sigma_0$ is a horizontal slice, contradicting $[\Sigma_0]=0$. Thus $C\neq0$, and in particular the curve does not meet the rotation axis. Reversing $s$ if necessary, write $C=c^{n-1}$ with $c>0$. Then
	\[
	z_s=\frac{c^{n-1}}{r^{n-1}}>0,
	\]
	so $z$ is a global parameter. If $a<b$ are the boundary heights, the curve can be written as $r=f(z)$ for $a\leq z\leq b$. Minimality, the inclusion $\Sigma_0\subset\overline{M_\rho}$, and the contact-angle condition yield \eqref{eq:band-catenoid-ode}, \eqref{eq:band-inside}, and \eqref{eq:band-angles}, respectively.
		
		\smallskip
		\textit{Step 3.}
		We verify that the rotational graph is an admissible catenoid band.
		By \eqref{eq:band-angles},
		\[
		\arctan f'(b)=\arctan\rho'(b)+\pi-\theta<\frac{\pi}{2},
		\]
		so
		\[
		\rho'(b)<\tan\left(\theta-\frac{\pi}{2}\right)=-\cot\theta.
		\]
		By continuity, $\rho'<-\cot\theta$ for $z>b$ sufficiently close to $b$, and hence $b<\ell_\theta$.  Therefore $\Sigma_0$ is an admissible catenoid band.
		
		\smallskip
		\textit{Step 4.}
		We prove $A_\theta(b)=B_\theta(a)$ when $\theta\neq\pi/2$.
		From \eqref{eq:band-angles},
		\begin{align*}
			\frac1{\sqrt{1+f'(a)^2}}
			&=\cos\bigl(\arctan\rho'(a)-(\pi-\theta)\bigr)\\
			&=\frac{-\cos\theta+\rho'(a)\sin\theta}{\sqrt{1+\rho'(a)^2}}
			=-\cos\theta\frac{1-\tan\theta\,\rho'(a)}{\sqrt{1+\rho'(a)^2}},\\
			\frac1{\sqrt{1+f'(b)^2}}
			&=\cos\bigl(\arctan\rho'(b)+(\pi-\theta)\bigr)\\
			&=\frac{-\cos\theta-\rho'(b)\sin\theta}{\sqrt{1+\rho'(b)^2}}
			=-\cos\theta\frac{1+\tan\theta\,\rho'(b)}{\sqrt{1+\rho'(b)^2}}.
		\end{align*}
		Using $f(a)=\rho(a)$ and $f(b)=\rho(b)$ in \eqref{eq:catenoid-first-integral}, we obtain
		\[
		c^{n-1}=-\cos\theta\,B_\theta(a)=-\cos\theta\,A_\theta(b).
		\]
		Since $\cos\theta\neq0$, it follows that
		\[
		A_\theta(b)=B_\theta(a).
		\]
	\end{proof}

	We now show that the hypotheses of Theorem~\ref{thm:intro-band-rigidity} force every boundary component to be a horizontal sphere.
	
	\begin{proof}[Proof of Theorem~\ref{thm:intro-band-rigidity}] 
		We first use \eqref{eq:main-nonnegative-profile} and Proposition~\ref{prop:bottom-disk-rigidity} to prove that $W_{\Sigma_0}\subset S_\rho$ for every component $\Sigma_0$ of $\Sigma$. For each such $\Sigma_0$, we then combine Proposition~\ref{prop:integral-identity} with \eqref{eq:main-nonnegative-profile} and \eqref{eq:no-cylinder-below-ell} to prove that every component of $\partial\Sigma_0$ is a horizontal sphere.
		
		Recall that
		\[
		J=\frac{\rho^{n-1}}{\sqrt{1+\rho'^2}}.
		\]
		Fix $a\in(0,L)$ and let $b\downarrow a$ in \eqref{eq:main-nonnegative-profile}. By continuity,
		\[
		0\leq A_\theta(a)-B_\theta(a)
		=2J(a)\tan\theta\,\rho'(a).
		\]
		Since $J>0$ and $\tan\theta<0$, we obtain $\rho'\leq0$ on $(0,L)$. As $-\cot\theta>0$, it follows that $\rho'<-\cot\theta$ on $(0,L)$, and hence $\ell_\theta=L$. In particular, \eqref{eq:slope-condition} holds automatically.
		
		Suppose that $D_0\subset W_{\Sigma_0}$ for some component $\Sigma_0$ of $\Sigma$. Proposition~\ref{prop:bottom-disk-rigidity} then implies that $\Sigma_0$ is a horizontal slice. Such a $n$-ball can meet $S_\rho$ at angle $\theta$ only at a height where $\rho'=-\cot\theta>0$, contradicting $\rho'\leq0$. Therefore $W_{\Sigma_0}\subset S_\rho$, equivalently,
		\[
		[\Sigma_0]=0
		\qquad\text{for every component $\Sigma_0$ of $\Sigma$}.
		\]
		
		Fix a component $\Sigma_0$ of $\Sigma$. We first show that every component of $\partial\Sigma_0$ is a horizontal sphere. Let $E$ be a component of $W_{\Sigma_0}$, and set
		\[
		\alpha=\min_E z,\qquad \beta=\max_E z,\qquad K_+(z)=\max_{\alpha\leq s\leq z}B_\theta(s),\qquad X_+=\frac{1-K_+/J}{-\tan\theta}\nabla^{S_\rho}z.
		\]
		Since $B_\theta\in C^1([\alpha,\beta])$, the function $K_+$ is Lipschitz and nondecreasing. By definition, $B_\theta(z)\leq K_+(z)$. For $\alpha\leq s<z\leq\beta$, \eqref{eq:main-nonnegative-profile} gives $B_\theta(s)\leq A_\theta(z)$, while for $s=z$,
		\[
		A_\theta(z)-B_\theta(z)=2J(z)\tan\theta\,\rho'(z)\geq0.
		\]
		Hence
		\[
		B_\theta(z)\leq K_+(z)\leq A_\theta(z).
		\]
		These inequalities are the reverse of those used in the proof of Proposition~\ref{prop:automatic-separation}. The same calculation, with $K_+$ in place of $K_-$, gives
		\[
		|X_+|\leq\frac{-\rho'}{\sqrt{1+\rho'^2}}.
		\]
		Let $\mu_E$ denote the outward unit conormal of $\partial E$ in $E$. Then
		\[
		\langle X_+,\mu_E\rangle\geq-|X_+|\geq\frac{\rho'}{\sqrt{1+\rho'^2}}
		\qquad\text{on }\partial E.
		\]
		A similar computation gives, almost everywhere,
		\[
		\Delta_{S_\rho}z+\tan\theta\,\operatorname{div}_{S_\rho}X_+
		=\frac{K_+'}{\rho^{n-1}\sqrt{1+\rho'^2}}\geq0,
		\]
		since $K_+$ is nondecreasing. Integrating over $E$ and applying the divergence theorem gives
		\begin{align}
			Q_E
			&:=\int_E\Delta_{S_\rho}z\,dA+\tan\theta\int_{\partial E}\frac{\rho'}{\sqrt{1+\rho'^2}}\,ds\nonumber\\
			&=\underbrace{\int_E\left(\Delta_{S_\rho}z+\tan\theta\,\operatorname{div}_{S_\rho}X_+\right)dA}_{=:I_E}
			+\underbrace{\tan\theta\int_{\partial E}\left(\frac{\rho'}{\sqrt{1+\rho'^2}}-\langle X_+,\mu_E\rangle\right)ds}_{=:B_E}\geq0.
		\end{align}
		The preceding inequalities and $\tan\theta<0$ give $I_E\geq0$ and $B_E\geq0$. Summing over all components $E$ of $W_{\Sigma_0}$, Proposition~\ref{prop:integral-identity} gives
		\[
		\sum_E Q_E=0.
		\]
		Hence $Q_E=0$ for every $E$, and therefore
		\[
		I_E=B_E=0.
		\]
		
		The integrand defining $B_E$ is continuous and nonnegative. Since $B_E=0$, it vanishes identically. Equality in the estimates used above therefore gives
		\[
		\langle X_+,\mu_E\rangle=-|X_+|=\frac{\rho'}{\sqrt{1+\rho'^2}}
		\qquad\text{on }\partial E.
		\]
		Thus, at any $p\in\partial E$ with $\rho'(z(p))\neq0$, we have
		\[
		\mu_E=-\frac{X_+}{|X_+|}\parallel\nabla^{S_\rho}z.
		\]
		Since $\mu_E\perp T_p\partial E$,
		\[
		dz|_{T_p\partial E}=0.
		\]
		
		Let $\Lambda$ be a component of $\partial\Sigma_0$ and suppose that $z|_\Lambda$ is nonconstant. By Sard's theorem, the regular values of $z|_\Lambda$ are dense in $(\min_\Lambda z,\max_\Lambda z)$. For such a regular value $t$,  there is $p\in\Lambda$ with $z(p)=t$ and $dz|_{T_p\Lambda}\neq0$, hence $\rho'(t)=0$. By continuity,
		\[
		\rho'\equiv0\qquad\text{on }[\min_\Lambda z,\max_\Lambda z],
		\]
		contradicting \eqref{eq:no-cylinder-below-ell}. Thus $z$ is constant on every component $\Lambda$ of $\partial\Sigma_0$. Since $\Lambda$ is open and closed in the horizontal sphere containing it, $\Lambda$ is the entire sphere.
		
		Lemma~\ref{lem:horizontal-boundary-catenoid} implies that $\Sigma_0$ is an admissible catenoid band. If its boundary heights are $a<b$, then
		\[
		A_\theta(b)=B_\theta(a).
		\]
		Since $\Sigma_0$ is arbitrary, the theorem follows.
	\end{proof}

	\appendix
	
	\section{A catenoid-band example}
	\label{app:equality-model}
	
	This appendix shows that \eqref{eq:two-point-barrier} cannot in general be omitted from Theorem~\ref{thm:intro-disk-rigidity}(iii) for $\pi/2<\theta<\pi$. We construct a smooth one-ended rotational support satisfying \eqref{eq:slope-condition} and carrying an admissible catenoid band with boundary heights $a<b$ such that
	\[
	A_\theta(b)=B_\theta(a).
	\] 
	Consequently, the conclusion of Theorem~\ref{thm:intro-disk-rigidity}(iii) may fail if \eqref{eq:two-point-barrier} is omitted, even when \eqref{eq:slope-condition} holds.
	
	\begin{proposition}
		\label{prop:profile-equality-nonrigid}
		For every $n\geq2$ and $\pi/2<\theta<\pi$, there exists a smooth one-ended rotational support $S_\rho$ without cylindrical intervals such that
		\[
		0<\rho'<-\cot\theta\qquad\text{on }[0,\infty),
		\]
		and $S_\rho$ supports an admissible catenoid band with boundary heights $a<b$ satisfying
		\[
		A_\theta(b)=B_\theta(a).
		\]
	\end{proposition}
	
	\begin{proof}
		Let $f$ be the radial function of an $n$-dimensional catenoid with neck at height zero. Then $f$ is even, $f'$ is odd, and $f'$ is strictly increasing. Choose
		\[
		0<u<\pi-\theta,\qquad 0<v<\theta-\frac{\pi}{2}
		\]
		sufficiently small. Choose heights $a<b$ on opposite sides of the neck such that
		\[
		\arctan f'(a)=-(\pi-\theta)+u,\qquad
		\arctan f'(b)=\pi-\theta+v.
		\]
		When $u=v=0$, the two heights are symmetric. For $u,v>0$, the lower height moves toward the neck and the upper height moves away from it. Hence $f(b)>f(a)$.
		Set
		\[
		m=\frac{f(b)-f(a)}{b-a}.
		\]
		Then $m>0$, and
		\[
		m\longrightarrow0
		\qquad\text{as }u,v\downarrow0.
		\]
		Since $-\cot\theta>0$, we may choose $u,v$ so that
		\[
		0<m<-\cot\theta,
		\qquad
		0<\tan u,\tan v<-\cot\theta.
		\]
		The choices of $a$ and $b$ also give
		\[
		f'(a)<0<\tan u,\qquad 0<\tan v<f'(b).
		\]
		After a vertical translation, assume $a>0$, with $a$ as small as needed below.
		
		For small $\varepsilon>0$, set
		\[
		s_\varepsilon(z)=
		\begin{cases}
			\tan u,&a\leq z\leq a+\varepsilon,\\
			m_\varepsilon,&a+\varepsilon<z<b-\varepsilon,\\
			\tan v,&b-\varepsilon\leq z\leq b,
		\end{cases}
		\qquad
		m_\varepsilon=\frac{m(b-a)-\varepsilon(\tan u+\tan v)}{b-a-2\varepsilon}.
		\]
		Then
		\begin{align*}
			\int_a^b s_\varepsilon(z)\,dz
			&=\varepsilon\tan u+(b-a-2\varepsilon)m_\varepsilon+\varepsilon\tan v\\
			&=\varepsilon(\tan u+\tan v)
			+m(b-a)-\varepsilon(\tan u+\tan v)\\
			&=m(b-a)  =f(b)-f(a).
		\end{align*}
		Moreover, $m_\varepsilon\to m$ as $\varepsilon\downarrow0$.
		For sufficiently small $\varepsilon$, all three values of $s_\varepsilon$ lie in a compact subinterval of $(0,-\cot\theta)$. Smooth the two jumps in disjoint intervals of length $\eta$, without changing $s_\varepsilon$ near $a$ or $b$, and denote the resulting function by $\widetilde s$. It may be chosen with the same upper and lower bounds as $s_\varepsilon$, and
		\[
		d_\eta:=f(b)-f(a)-\int_a^b\widetilde s(z)\,dz\longrightarrow0
		\qquad\text{as }\eta\downarrow0.
		\]
		Choose $\psi\in C_c^\infty((a,b))$ supported away from the two smoothed intervals, with
		\[
		\int_a^b\psi\,dz=1.
		\]
		Then
		\[
		s=\widetilde s+d_\eta\psi
		\]
		is smooth and has the prescribed integral. For sufficiently small $\eta$, the correction does not leave $(0,-\cot\theta)$. Thus
		\[
		s=\tan u,\ \ \text{near }a,\qquad s=\tan v,\ \ \text{near }b,\qquad
		0<s<-\cot\theta,
		\]
		and
		\[
		\int_a^b s(z)\,dz=f(b)-f(a).
		\]
		Choose $\eta=\eta(\varepsilon)$ so small that
		\[
		\|s-s_\varepsilon\|_{L^1(a,b)}\leq\varepsilon.
		\]
		
		Define
		\[
		\rho(z)=f(a)+\int_a^z s(\zeta)\,d\zeta
		\qquad(a\leq z\leq b).
		\]
		Then $\rho(a)=f(a)$ and $\rho(b)=f(b)$. Since $f'(a)<0<s$ near $a$ and $s<f'(b)$ near $b$, we have $\rho>f$ near $a$ and $b$. On the remaining compact subinterval, the chord joining $(a,f(a))$ and $(b,f(b))$ has positive distance from the strictly convex graph of $f$, while $\rho$ converges uniformly to this chord as $\varepsilon\downarrow0$. Hence, for sufficiently small $\varepsilon$,
		\[
		\rho(z)>f(z)\qquad(a<z<b).
		\]
		
		Since $s$ is constant near $a$ and $b$, extend $\rho$ smoothly by
		\[
		\rho(z)=f(a)+\tan u\,(z-a)\qquad(0\leq z\leq a),
		\]
		and
		\[
		\rho(z)=f(b)+\tan v\,(z-b)\qquad(z\geq b).
		\]
		Taking $a$ sufficiently small gives $\rho(0)=f(a)-a\tan u>0$. By construction,
		\[
		0<\rho'<-\cot\theta\qquad\text{on }[0,\infty).
		\]
		Thus $S_\rho$ has no cylindrical intervals. Moreover, there is no height $a$ with $\rho'(a)>-\cot\theta$, so \eqref{eq:slope-condition} holds. Since $\rho'\equiv\tan v>0$ for $z\geq b$, the support is complete, proper, one-ended, and has no upper cap.
		
		It remains to verify the contact angle. Since $\rho'(a)=\tan u$ and $\rho'(b)=\tan v$,
		\begin{align*}
			\arctan\rho'(a)-\arctan f'(a)
			&=u-\bigl(-(\pi-\theta)+u\bigr)=\pi-\theta,\\
			\arctan\rho'(b)-\arctan f'(b)
			&=v-\bigl(\pi-\theta+v\bigr)=-(\pi-\theta).
		\end{align*}
		Together with $\rho=f$ at $a,b$ and $\rho>f$ on $(a,b)$, this shows that the catenoid segment is an admissible $\theta$-capillary catenoid band.
		
		Finally, let $c$ be the constant in \eqref{eq:catenoid-first-integral}. Then
		\begin{align*}
			\frac1{\sqrt{1+f'(a)^2}}
			&=\cos\bigl(\arctan\rho'(a)-(\pi-\theta)\bigr) =-\cos\theta\,
			\frac{1-\tan\theta\,\rho'(a)}{\sqrt{1+\rho'(a)^2}},\\
			\frac1{\sqrt{1+f'(b)^2}}
			&=\cos\bigl(\arctan\rho'(b)+(\pi-\theta)\bigr) =-\cos\theta\,
			\frac{1+\tan\theta\,\rho'(b)}{\sqrt{1+\rho'(b)^2}}.
		\end{align*}
		Using $f(a)=\rho(a)$ and $f(b)=\rho(b)$ in \eqref{eq:catenoid-first-integral},
		\[
		c^{n-1}=-\cos\theta\,B_\theta(a)
		=-\cos\theta\,A_\theta(b).
		\]
		Since $\cos\theta\neq0$, we obtain
		\[
		A_\theta(b)=B_\theta(a).
		\]
	\end{proof}

	\section{Equivalent conditions on the support profile}
	\label{app:angle-uniform}
	
	This appendix characterizes the support profiles for which \eqref{eq:two-point-barrier} holds for every obtuse contact angle. Assume $\rho'>0$. For $0<\theta<\pi/2$, one has $\ell_\theta=0$, so there are no heights $0<a<b<\ell_\theta$; moreover, \eqref{eq:slope-condition} holds automatically for $0<\theta\leq\pi/2$. It therefore remains to consider $\pi/2<\theta<\pi$. Set
	\[
	J=\frac{\rho^{n-1}}{\sqrt{1+\rho'^2}}.
	\]
	
	The following proposition replaces the family of strict inequalities \eqref{eq:two-point-barrier}, one for each obtuse contact angle, by a single differential inequality for $\rho$. It also shows that this differential inequality implies \eqref{eq:slope-condition} for every contact angle.

	\begin{proposition}
		\label{prop:angle-uniform-limit}
		Assume $\rho'>0$. Then
		\[
		\begin{aligned}
			&\eqref{eq:two-point-barrier}\text{ holds for every }\pi/2<\theta<\pi\\
			&\quad\Longleftrightarrow\quad
			J'\leq0
			\quad\Longleftrightarrow\quad
			\Delta_{S_\rho}z\leq0
			\quad\Longleftrightarrow\quad
			\rho\rho''\geq(n-1)(1+\rho'^2).
		\end{aligned}
		\]
		Each of these conditions implies \eqref{eq:slope-condition} for every $0<\theta<\pi$. In addition,
		\[
		J'\equiv0
		\quad\Longleftrightarrow\quad
		\rho\rho''=(n-1)(1+\rho'^2),
		\]
		which is the rotational minimal hypersurface equation.
	\end{proposition}
	
	\begin{proof}
		Suppose first that $J'\leq0$, and fix $\pi/2<\theta<\pi$. Since $\tan\theta<0$ and $\rho'>0$,
		\[
		A_\theta<J<B_\theta
		\]
		at every height. For $0<a<b<\ell_\theta$, the function $J$ is nonincreasing, and hence
		\[
		A_\theta(b)<J(b)\leq J(a)<B_\theta(a),
		\]
		which proves \eqref{eq:two-point-barrier} for every $\pi/2<\theta<\pi$.

		Conversely, assume that \eqref{eq:two-point-barrier} holds for every $\pi/2<\theta<\pi$. If $J(b)>J(a)$ for some $0<a<b<L$, then for all sufficiently small $t>0$,
		\[
		J(b)(1-t\rho'(b))>J(a)(1+t\rho'(a)),\qquad \rho'(b)<\frac1t.
		\]
		Set $\theta=\pi-\arctan t$. Then $\tan\theta=-t$ and $-\cot\theta=1/t$. By continuity, $\rho'<1/t$ immediately above $b$, so $b<\ell_\theta$. Hence \eqref{eq:two-point-barrier} applies to $a<b$, but
		\[
		A_\theta(b)=J(b)(1-t\rho'(b))
		>J(a)(1+t\rho'(a))=B_\theta(a),
		\]
		a contradiction. Therefore $J$ is nonincreasing, and $J'\leq0$.
		
		Next,
		\[
		\Delta_{S_\rho}z
		=\frac{1}{\rho^{n-1}\sqrt{1+\rho'^2}}
		\frac{d}{dz}\left(\frac{\rho^{n-1}}{\sqrt{1+\rho'^2}}\right)
		=\frac{J'}{\rho^{n-1}\sqrt{1+\rho'^2}},
		\]
		so $J'\leq0$ is equivalent to $\Delta_{S_\rho}z\leq0$. Also,
		\[
		\frac{J'}{J}
		=\frac{(n-1)\rho'}{\rho}-\frac{\rho'\rho''}{1+\rho'^2}
		=\frac{\rho'}{\rho(1+\rho'^2)}
		\bigl((n-1)(1+\rho'^2)-\rho\rho''\bigr).
		\]
		Since $\rho'>0$,
		\[
		J'\leq0
		\quad\Longleftrightarrow\quad
		\rho\rho''\geq(n-1)(1+\rho'^2).
		\]
		The latter implies $\rho''>0$, and hence \eqref{eq:slope-condition} holds for every $0<\theta<\pi$.
		Finally, the formula for $J'/J$ gives
		\[
		J'\equiv0
		\quad\Longleftrightarrow\quad
		\rho\rho''=(n-1)(1+\rho'^2).
		\]
	\end{proof}

\end{document}